\documentclass[12pt,a4paper,oneside,reqno]{amsart}
\usepackage{amssymb}
\usepackage{amscd}
\numberwithin{equation}{section}

\theoremstyle{plain}

\newtheorem{theorem}[subsection]{Theorem}
\newtheorem{proposition}[subsection]{Proposition}
\newtheorem{lemma}[subsection]{Lemma}
\newtheorem{corollary}[subsection]{Corollary}
\newtheorem{conjecture}[subsection]{Conjecture}

\theoremstyle{definition}

\newtheorem{definition}[subsection]{Definition}

\newcommand\cN{{\mathcal N}}

\newcommand{\beq}{\begin{equation*}}
\newcommand{\eeq}{\end{equation*}}

\newcommand{\cO}{{\mathcal O}}
\newcommand{\cA}{{\mathcal A}}

\def\ox{\otimes}
\def\o+{\oplus}

\def\ua{\uparrow}
\def\Lra{\Longrightarrow}

\def\beqa{\begin{eqnarray}}
\def\eeqa{\end{eqnarray}}
                       
\newcommand{\al}{\alpha}

\newcommand{\la}{\lambda}
\newcommand{\si}{\sigma}

\newcommand{\om}{\omega}

\begin{document}

\title[]{{On the Existence of Stable bundles with prescribed Chern classes on Calabi-Yau threefolds}}
\author{Bj\"orn Andreas}
\address{Institut f\"ur Mathematik, Freie Universit\"at Berlin, Arnimallee 3, 14195 Berlin.}
\email{andreasb@mi.fu-berlin.de}
\thanks{B. A. is supported by DFG SFB 647: Space-Time-Matter. Analytic and Geometric Structures.}
\author{Gottfried Curio}
\address{Arnold-Sommerfeld-Center 
for Theoretical Physics, Department f\"ur Physik, 
Ludwig-Maximilians-Universit\"at M\"unchen, Theresienstr. 37, 80333 M\"unchen.}
\email{Gottfried.Curio@physik.uni-muenchen.de}
\thanks{G. C. is supported by DFG grant CU 191/1-1; ASC Report-Nr. 
LMU-ASC 09/11}

\begin{abstract}
We prove a case of the conjecture of Douglas, Reinbacher and Yau about the 
existence of stable vector bundles with prescribed Chern classes on a Calabi-Yau threefold.
For this purpose we prove the existence of certain stable vector bundle extensions
over elliptically fibered Calabi-Yau threefolds.
\end{abstract}


\maketitle

\section{Introduction}
The present note is concerned with the question of existence 
of stable bundles $V$ with prescribed Chern class $c_2(V)$ 
on a given Calabi-Yau threefold $X$.
The 'DRY'-conjecture of Douglas, Reinbacher and Yau in [\ref{DRY}] 
gives a sufficient condition for cohomology classes on
$X$ to be equal to
the Chern classes of a stable sheaf $V$. In this note we consider the case that $V$ is actually a vector bundle and $c_1(V)=0$.
In [\ref{DRYAC}] we showed
that infinitely many classes on an $X$ which is elliptically fibered over a base
surface $B$ exist for which the conjecture is true. 
A weak form of the conjecture (considered already 
in [\ref{DRYAC}])
asserts the existence of $V$ with a suitable prescribed second Chern class 
$c$.
In [\ref{DRY2}] we showed that rank $4$ {\em polystable} vector bundles 
exist for such suitable cohomology classes $c$ ("DRY-classes", cf. the definition below). 
In this note we prove the weak form of the conjecture
by providing corresponding {\em stable} bundles $V$, for all
ranks $\cN\geq 4$, for $B$ a Hirzebruch surface ${\bf F_g}$ or a del Pezzo surface ${\bf dP_k}$, in all but finitely many cases.

We recall the following definition (for use in this note) and the weak DRY-conjecture
\begin{definition}
 Let $X$ be a Calabi-Yau threefold of $\pi_1(X)=0$ and
$c\in H^4(X, {\bf Z})$, 
\begin{enumerate}
\item $c$ is called a {\em Chern class} 
if a stable $SU({\cN})$ vector bundle $V$ on $X$ exists with $c=c_2(V)$,
\item  $c$ is called a {\em DRY class} if an ample class 
$H\in H^2(X, {\bf R})$ exists (and an integer ${\cN}$) with 
\beqa
\label{weak dry}
c&=&{\cN}\Bigg( H^2+\frac{c_2(X)}{24}\Bigg).
\eeqa
\end{enumerate}
\end{definition}

\begin{conjecture} On a Calabi-Yau threefold $X$ with $\pi_1(X)=0$ 
every DRY class $c\in H^4(X, {\bf Z})$ is a Chern class.
\end{conjecture}
Here it is understood that the integer ${\cN}$ occurring in the two
definitions is the same.

In this paper we will construct a class of stable bundle extensions on an elliptically fibered Calabi-Yau threefold $\pi\colon X\to B$ with section $\sigma$ (we also denote by $\sigma$ the image divisor in $X$ and its cohomology class). We will consider $B$ to be a surface with ample $K_B^{-1}$ such as the Hirzebruch surface ${\bf F_g}$ with $g=0,1$ or the del Pezzo surface $\bf dP_k$ with $k=0,\dots, 8$ (note that ${\bf dP_1}\cong {\bf F_1}$).

The main result of this note is (using the decomposition $H^4(X,{\bf Z})\cong H^2(B,{\bf Z})\sigma\oplus H^4(B,{\bf Z})$; we usually identify $H^4(B,{\bf Z})$ with ${\bf Z}$).
\begin{theorem}\label{main}
Let the class $c=\phi \sigma+\omega$ 
be a DRY class. Then $c$ is a Chern class
\begin{enumerate}
\item up to finitely many exceptions in $c$
\begin{enumerate}
\item for $\cN\geq 4$ and $B={\bf F_0}$ ,
\item for $\cN\geq 6$ and $B={\bf dP_k}$ for $k=1,\dots, 8$
\end{enumerate}
\item without any exceptions in $c$
\begin{enumerate}
\item for $\cN\equiv 2 \; ({\rm mod}\: 4)$ with $\cN\neq 2$ and either $B={\bf F_g}$ or $B={\bf dP_k}$ with $k=1,\dots, 6$,
\item for $\cN\equiv 0 \; ({\rm mod}\: 4)$ and $B={\bf F_0}$.
\end{enumerate}
\end{enumerate}
\end{theorem}

This note has two parts. In section 2 we prove the existence of a stable
bundle extensions under certain conditions (equations (\ref{non-split condition})-(\ref{stab})) on the input data.

In section 3 we apply this to the weak DRY conjecture
by showing that, for a class of cases
described precisely below, a DRY class fulfills the assumptions of section 2. 

\section{Stable Bundle Extensions}
In this section we will construct a class of stable bundle extensions which will later, in section 3,
serve as the bundles which realize a given DRY class as Chern class. 

We consider the following extension
\beqa\label{extv}
0\to \pi^*E\otimes {\cO}_X(-nD)\to V\to W\otimes {\cO}_X(rD)\to 0
\eeqa
where $E$ is a stable rank $r$ vector bundle on $B$ with Chern classes $c_1(E)=0$; 
these bundles exist on rational surfaces if $c_2(E)\geq r+2$ [\ref{Art}];
$D=\pi^*\alpha$ is a divisor in $X$ with $\alpha$ a divisor in $B$ and $W$ a rank $n$ spectral cover bundle with $c_1(W)=0$. 
Let $C$ be an irreducible surface in the linear system $|n\sigma+\pi^*\eta|$ (where we denote by $\eta$ 
a divisor class in $B$ and likewise its cohomology class) and $i\colon C\to X$ the immersion of $C$ into $X$ and let $L$ be a rank one sheaf on $C$. We say $W$ is a spectral cover bundle [\ref{FMW}] of rank $n$ if $W=\pi_{1*}(\pi_2^*(i_*L)\otimes \mathcal{P})$ where $\mathcal{P}$ is the Poincar\'e sheaf on the fiber product $X\times_B X$ and $\pi_{1,2}$ are the respective projections on the first and second factor. The condition 
$c_1(W)=0$ leads to (here and in the sequel $c_1$ denotes $c_1(B)$) (cf. [\ref{FMW}])
\begin{equation}
c_1(L)=n\Big(\frac{1}{2}+\lambda\Big)\sigma+\Big(\frac{1}{2}-\lambda\Big)\pi^*\eta+\Big(\frac{1}{2}+n\lambda\Big)\pi^*c_1.
\end{equation}
Since $c_1(L)$ must be an integer class it follows that: if $n$ is odd, then $\lambda$ is strictly half-integral and for $n$ even an integral $\lambda$ requires
$\eta\equiv c_1 \  ({\rm mod}\  2)$ while a strictly
half-integral $\lambda$ requires $c_1$ even. Moreover note that, to assure that the linear system $|n\sigma+\pi^*\eta|$ contains an irreducible surface $C$ it is sufficient to demand that the linear system $|\eta|$ is base-point free in $B$ and that the divisor corresponding to the cohomology class $\eta-nc_1$ is effective [\ref{FMW}].

Now let $H_0$ and $H_B$ be fixed ample divisors in $X$ and $B$, respectively. We have that 
$\pi^*E$ and $W$ are stable with respect to $H=\epsilon H_0+\pi^*H_B$ for $\epsilon>0$ chosen 
sufficiently small. For $W$ this is due to Theorem 7.2. in [\ref{FMW3}] and for $\pi^*E$ Theorem 3.1 in [\ref{AGF2}].

Two necessary conditions for $V$ to be stable are:
\begin{enumerate}
\item $DH^2>0$, 
\item ${Ext}^1(W\otimes {\cO}_X(rD), \pi^*E\otimes {\cO}_X(-nD))\neq 0$.
\end{enumerate}
Here the first condition assures that $\pi^*E\ox\cO_X(-nD)$ is not a destabilizing subbundle of $V$
and the second condition (which is equivalent to $H^1(X, \pi^*E\otimes W^*\otimes \cO_X(-mD))\neq 0$ 
with $m=r+n=\cN$) assures that a non-split extension exists and shows that $W\otimes \cO(X(rD))$ is not a
destabilizing subbundle $V$. 

\begin{theorem}
Assume (i) and (ii) are satisfied and $\alpha\cdot H_B=0$. Then $V$ as defined in (2.1)
is stable with respect to $H=\epsilon H_0+\pi^*H_B$ for sufficiently small $\epsilon>0$.
\end{theorem}
\begin{proof}
To prove stability of $V$ consider the diagram of exact sequences \[ \begin{array}{ccccccccc}
  &  &0 & & 0   &  &0 & \\
   &      & \ua &                   &\ua&             & \ua &      &\\
0 &\to &  \pi^*E/F_s\ox \cO_X(-nD) & \to & V/V'_{s+t}&\to & W/G_t\ox\cO_X(rD)&\to &0\\
  &      & \ua &                   &\ua&                    & \ua &      &\\
0 & \to & \pi^*E\ox \cO_X(-nD)& \stackrel{i}{\to} & V & \stackrel{j}{\to}
 &W \ox\cO_X(rD)& \to & 0\\
  &      & \ua &                   &\ua&                    & \ua &      &\\
0 & \to & F_s\ox \cO_X(-nD)& {\to}& V'_{s+t} & {\to} & G_t\ox\cO_X(rD) & \to & 0\\
 &      & \ua &                   &\ua&             & \ua &      &\\
 &  &0 & & 0   &  &0 & \\ \end{array} \]
with $F_s\ox \cO_X(-nD)=i^{-1} V'_{s+t}$ and  $G_t\ox \cO_X(nD)=j(V'_{s+t})$ of ranks $0\leq s\leq r$ and $0\leq t\leq n$ for a subsheaf $V'_{s+t}$ of $V$. 

Note that $s=0$ or $t=0$ implies $F_s=0$ or $G_t=0$, respectively. Moreover, note for $0<s< r$ we have $c_1(F_s)=-A_1\sigma +\pi^*\lambda$ with $A_1\geq 0$ and $\lambda\cdot H_B<0$. To see this consider a
$F_s$, a subsheaf of $\pi^*E$ where we can assume that $\pi^*E/F_s$ is torsion free.
We have $0\to F_s\vert_{\sigma}\to E$ and $c_1(F_s\vert_{\sigma})H<0$. Similarly we get for restriction to the fiber $F$ that $0\to F_s\vert_{F}\to \cO^r_F$
thus $deg(F_s\vert_{F})\leq 0$ as $\cO^r_F$ is semistable. Thus $A_1\geq 0$ and $\la H_B < 0$.

For spectral cover bundles and $0<t< n$ we have $c_1(G_t)=-A_2\sigma +\pi^*\beta$ with $A_2> 0$ (Theorem 7.2, [\ref{FMW3}]).

We need to show for all subsheaves $V'_{r+s}$ of $V$ with $0\leq s\leq r$ and $0\leq t\leq n$ and $0<s+t<n+r$ that $\mu(V'_{r+s})<0$. We can assume that the quotient $V/V'_{r+s}$ is torsion free thus all subsheaves $V'_{s+t}$ with $s=r$ need not be considered since $\pi^*E/F_r$ is a torsion sheaf thus $0\leq s< r$.

i) For subheaves $V'_{s+t}$ with $0\leq s<r$ and $0<t<n$ the slope 
is given by 
$$(s+t)\mu(V'_{s+t})\leq -H_B^2\sigma +O(\epsilon)<0$$
where the latter inequality holds for $\epsilon$ sufficiently small. If $s=0$ and 
$0<t<n$ we find similarly $t\mu(V'_{0+t})=-H_B^2\sigma +O(\epsilon)$ and $\mu(V'_{0+t})<0$ for 
$\epsilon$ sufficiently small.

ii) For subsheaves $V'_{s+t}$ with $0< s<r$ and $t=0$ we get
$$\mu(V'_{s})= -nDH^2+\mu(F_s)$$
but since $DH^2>0$ by assumption and $\mu(F_s)<0$ by stability of $\pi^*E$ we get $\mu(V'_s)<0$.

So we are left with the cases where $t=n$. In this case $W/G_n$ is a torsion sheaf. Let us write $c_1(W/G_n)={\bar D}$; since $W/G_n$ is a torsion sheaf one has ${\bar D}\cdot H^2\geq 0$ and ${\bar D}\cdot H^2=0$ if and only if $W/G_n$ is supported in codimension $\geq 2$. Moreover, ${\bar D}$ is an effective divisor in $X$.
From $G_n\to W$, we get the map $(\Lambda^nG_n)^{**}\to\Lambda^nW^{**}$  and that moreover $rk(\Lambda^nG_n)=rk(\Lambda^nW)=1$ and $(\Lambda^nG_n)^{**}$ is a reflexive torsion free sheaf of rank one with $c_1(G_n)=c_1((\Lambda^nG_n)^{**})$. Furthermore, we have $(\Lambda^nW)^{**}=L$ with $c_1(L)=0$. Thus $c_1(G_n)=-{\bar D}$ with ${\bar D}$ effective.

iii) For subheaves $V'_{s+t}$ with $0<s<r$ and $t=n$ we get for $\epsilon$ sufficiently small
$$(s+n)\mu(V'_{s+n})\leq\epsilon \big(2H_0\pi^*(\lambda H_B)+ \epsilon H_0^2(\pi^*\lambda +n(r-s)\pi^*\alpha)\big)<0$$
since $\lambda H_B<0$ and $H_0^2\lambda$ is bounded.

iv) For subheaves $V'_{s+n}$ with $s=0$ and $t=n$ we find
$$\mu(V'_{0+n})= r DH^2-{\bar D}H^2$$
For ${\bar D}H^2>0$ we get (with ${\bar D}=y\sigma+\pi^*{\bar\beta}$) 
$$\mu(V'_{0+n})\leq \epsilon\big(-2H_0\pi^*({\bar\beta}H_B)+\epsilon rH_0^2\pi^*\alpha\big)$$
and for $\epsilon$ sufficiently small that $\mu(V'_{0+n})<0.$  

It remains to prove that the case ${\bar D}H^2=0$ cannot possibly happen.
Note that $V$ does not admit a destabilizing subsheaf $V'_{0+n}=G_n\ox \cO_X(rD)$
if the map 
$$f\colon Ext^1(W\otimes \cO_X(rD), \pi^*E\otimes \cO_X(-nD))\to
Ext^1(G_n\otimes \cO_X(rD),\pi^*E\otimes \cO_X(-nD))$$ 
is injective ([\ref{AC1}], Lemma 2.3). To see this in our case, we apply $Hom(\,\,\, ,\pi^*E\otimes \cO_X(-nD))$
to the exact sequence
$$0\to G_n\ox \cO_X(rD)\to W\ox\cO_X(rD)\to W/G_n\ox \cO_X(rD)\to 0$$
to obtain (set ${\mathcal E}=\pi^*E\otimes \cO_X(-nD)$ and $T=W/G_n$)
$$Ext^1(T\otimes\cO_X(rD),{\mathcal E})\to Ext^1(W\otimes \cO_X(rD), {\mathcal E})\to
Ext^1(G_n\otimes \cO_X(rD),{\mathcal E})$$
Since $T$ is supported in codimension $\geq 2$, by duality $Ext^1(T\otimes\cO_X(rD),{\mathcal E})=
Ext^2({\mathcal E}, T\ox\cO_X(rD)\ox\cO_X(K_X))^*=Ext^2({\mathcal E}, T\ox\cO_X(rD))=H^2(X,{\mathcal E}^*\ox T\ox\cO_X(rD))^*=0$ so that $f$ is injective, and this finishes the proof.
\end{proof}
 
Finally let us analyze when the above extension can be chosen non-split. 
To simplify notation set $F=E\otimes \cO_B(-m\alpha)$. The Hirzebruch-Riemann-Roch theorem on $X$ gives
\beqa
I_X=\sum_{i=0}^3(-1)^i \dim H^i(X, \pi^*F\otimes W^*)&=&\int_X ch(\pi^*F\otimes W^*)Td(X)\nonumber\\
&=&r\big(-\lambda\eta+m\alpha\big)(\eta-nc_1).
\eeqa
The following lemma will be helpful for proving the next proposition.
\begin{lemma}
Let $W$ be a spectral cover bundle over $X$ then 
\begin{enumerate}
\item $\pi_*W=\pi_*W^*=0,$
\item the sheaves $R^1p_*W$ and $R^1\pi_*W^*$ are supported on $\cA=C\cap \sigma$.
\end{enumerate}
\end{lemma}
\begin{proof}
(i): For a given spectral cover bundle $W$
one has $\pi_*W=0$. At a generic point $b\in B$ one has the stalk
$(\pi_*W)_b=H^0(F, W|_F)=\bigoplus_{i=1}^nH^0(F,{\cO}_F(q_i-p))$
where $p=\si F$ is the zero element in the group law on the fibre
$F$ over $b\in B$ and $q_i$ are the points at which the spectral cover
of $V_n$ intersects $F$.
Now ${\cO}(q_i-p)$ is generically a non-trivial bundle of degree zero which
over an elliptic curve admits no global sections.
Thus $H^0(F,{\cO}_F(q_i-p))=0$ for all $i$ and so $(\pi_*W)|_b=0$.
However, since $V_n$ is torsion free, $\pi_*W$ is also torsion free.
Thus $(\pi_*W)|_b=0$ for generic $b\in B$ gives $\pi_*W=0$ everywhere. As $W^*$ is again a spectral cover bundle one has also $\pi_*W^*=0$.\\
(ii): For generic points $b\in B$ where $p$ is distinct from the $n$ points $p_i$ we have 
$(R^1\pi_*W)_b=0$. However, at points in $B$ at which $p$ is equal to one of the points $p_i$ 
we find that $(R^1\pi_*W)_b\neq 0$. The locus of such points form a curve $C\cap \sigma$ in $B$. 
\end{proof}
\begin{proposition}
$H^3(X,\pi^*F\otimes W^*)=0.$
\end{proposition}
\begin{proof}
The Leray spectral sequence applied to $\pi\colon X\to B$ 
leads to (using the fact that $\pi_*W^*=0$)
\beqa
H^3(X,\pi^*F\otimes W^*)\cong H^2(B, R^1\pi_*W^*\otimes F)
\eeqa
as $R^1\pi_*W^*$ is a torsion sheaf on $B$ supported on the curve of class $\cA=C\cap \sigma$
we get $H^2(B, R^1\pi_*W^*\otimes F)=0$ and conclude.
\end{proof}
Thus imposing $I_X<0$ one gets a non-split extension (\ref{extv}) as $H^1(X, \pi^*F\otimes W^*)\neq 0$.

In summary, we get the following list of conditions
(where $\alpha \neq 0$ because of (\ref{stab})):
\beqa
\label{non-split condition}
\Big[ \lambda\eta-m\alpha\Big] (\eta-nc_1)&>&0\\
\label{orthogonality condition}
\alpha H_B&=&0\;\;\;\;\;\;\;\;\;\;\;\;\;\;\;\;\;\;\; 
(\Lra \alpha \neq \pm \mbox{effective})\;\;\;\;\;\;\\
\pi^*\alpha H_0^2&>&0\label{stab}
\eeqa
To solve (\ref{stab}) we note
\begin{lemma}
Let $X$ be an elliptically fibered Calabi-Yau threefold with section $\sigma$, then $H_0=x\sigma +\pi^*\rho$ 
is ample if and if $x>0$ and $\rho-xc_1$ is ample in $B$.
\end{lemma}
The proof of this Lemma is given in appendix A in [\ref{DRYAC}].
Thus condition (\ref{stab}) becomes
\beqa
(2\rho -xc_1)\alpha>0.\label{condh}
\eeqa
In the following, (\ref{condh}) will not be considered as a condition on $\alpha$ but rather
as a condition on $x$ and $\rho$ (for each respective $\alpha$). To show that it is possible to solve this for $x$ and $\rho$
note that it is enough to show that an ample class $h$ exists
with $h \alpha >0$: on the one hand the expression in brackets in (\ref{condh})
$h:=\rho+(\rho-xc_1)$ is such a class, and on the other hand
any $h$ can be written in such a way\footnote{one has
$h=(A+xc_1)+A=2[ A + \frac{x}{2}c_1]$ with the ample class
$A:=\rho - x c_1$; conversely one can, if an ample $h$ is given,
solve this for an ample class $A$ as $\frac{1}{2}h$ will be also ample
and thus also $A=\frac{1}{2}h-\frac{x}{2}c_1$ for $x>0$
sufficiently small (the ample cone is an open set)}.
  Now the existence of a nonzero $\alpha$ which is neither
effective nor anti-effective, cf.~equ.~(2.6), presupposes
that $h^{1,1}(B)\geq 2$. Let $\alpha=\beta - \gamma$
be a decomposition with $\beta$ and $\gamma$ effective,
then it is possible to choose in the open ample cone an element $h$
with $h\beta = 2 $ and $h\gamma =1$.

Finally let us give the expressions for the second Chern classes 
of $V$ and $W$
\beqa
c_2(V)&=&c_2(W)+c_2(\pi^*E)-\frac{rn(r+n)}{2}\pi^*\alpha^2,\\
c_2(W)&=&\pi^*\eta\cdot\sigma-\frac{n^3-n}{24}\pi^*c_1^2+(\lambda^2-\frac{1}{4})\frac{n}{2}\pi^*\eta\cdot(\pi^*\eta-n\pi^*c_1).
\eeqa
We also note that for the special case $n=r$ a simplified version of the construction
is possible where the twisting is as follows
\beqa\label{extv2}
0\to \pi^*E\otimes {\cO}_X(-D)\to V\to W\otimes {\cO}_X(D)\to 0
\eeqa
and the stability proof of $V$ is complete analogous. In this case we get
\beqa 
c_2(V)&=&c_2(W)+c_2(\pi^*E)-\pi^*\alpha^2.
\eeqa

\section{Proof of main result}
The aim of this section is to prove Theorem \ref{main}. For this let us first recall
a result characterizing DRY classes on elliptically fibered Calabi-Yau threefolds
obtained in [\ref{DRYAC}] 

\begin{theorem}
A class $c=\phi\si + \omega \in H^4(X, {\bf Z})$ is a DRY class if and 
only if the following condition is fulfilled
(where $b$ is some $b\in{\bf R}^{>0}$ and 
$\omega \in H^4(B, {\bf Z})\cong {\bf Z}$):\\
$\phi-{\cN}(\frac{1}{2}+b)c_1$ is ample and 
$\frac{1}{\cN}\omega> \omega_0(\phi; b):=R+\frac{c_1^2}{4}(b+\frac{q}{b})$.
\end{theorem}
Here we use the following abbreviations (cf.~[\ref{DRYAC}]):
$R:=\frac{1}{2{\cN}}\phi c_1+\frac{1}{6}c_1^2+\frac{1}{2}$,
$q:=\frac{(\phi - \frac{{\cN}}{2}c_1)^2}{{\cN}^2c_1^2}$.
Note that under the hypothesis that 
$\phi - \frac{{\cN}}{2}c_1=A+b{\cN}c_1$ with an ample class $A$ 
on $B$ (and with $c_1\neq 0$ effective) one has 
$(\phi - \frac{{\cN}}{2}c_1)^2>b^2{\cN}^2c_1^2$,
thus one has $b<\sqrt{q}$. On the other hand $\om_0$ reaches its minimum
as a function of $b$ for $b=\sqrt{q}$, 
giving that $\om_0(\phi)>R+\frac{c_1^2}{2}\sqrt{q}
=R+\frac{1}{2\cN^2}\sqrt{c_1^2}\sqrt{(\phi-\frac{\cN}{2}c_1)^2}$.

Important for us will be especially the following 
(recall that we assume that $c_1$ is ample)

\begin{corollary}
For a DRY class $c=\phi\si + \omega$ one has that
\begin{enumerate}
\item $\phi-\frac{{\cN}}{2}c_1$ is ample,\\
\item $\om>
\cN\Big(R+\frac{1}{2\cN^2}\sqrt{c_1^2}\sqrt{(\phi-\frac{\cN}{2}c_1)^2}\Big)$.
\end{enumerate}
\end{corollary}

Our goal is to show that, for a certain class of cases the conditions 
for a DRY class $c$ imply indeed the
conditions under which the stable extension bundle $V$ of the previous section
exist (the latter will have the prescribed class $c$ as $c_2(V)$).

Before we come to the general discussion of all cases with $\cN\geq 4$ 
we have to mention two special cases.
This is caused by a restriction in the relation between the ranks $n$ and $r$
of the spectral bundle and the pullback bundle, respectively.
From the Corollary we know given a DRY class $c$, that $\eta-\frac{n+r}{2}c_1$ is 
effective (we will put $\eta:=\phi$); on the other hand for the spectral
construction we need that $\eta - n c_1$ is effective. Thus one is lead to 
the condition
\beqa
r\geq n
\eeqa

\subsection{The cases $\cN=4, 5$ and  $B\neq {\bf F_0}$}

In this case one would have to assume that $n$ is even (actually $n=2$)
and so that $\eta\equiv c_1 (mod \; 2)$; thus
one would not be sufficiently general in the choice of $\eta$ to match a 
given $\phi$ in the DRY class $c=\phi \si + \omega$. Thus, if one is forced
to have $n$ even, we can only cover the case $B={\bf F_0}$.
On the other hand 
this will be no problem, of course, if $\cN=n+r\geq 6$ as one can then
choose $n=3$. For $\cN=4,5$, however, we can only cover the case $B={\bf F_0}$.

\subsection{The cases $\cN\geq 6$ or $\cN\geq 4$ and  $B={\bf F_0}$}

Let us therefore assume from now on that we are in the case that one has either 
$\cN\geq 6$, and thus $n$ can be choosen odd (actually $n=3$), 
or that $\cN=4$, and thus $n$ is even (actually $n=2$), but $B= {\bf F_0}$.
In both cases we can then use a strictly half-integral $\la$; we will 
have to choose $\la = \pm 1/2$ to kill the corresponding contribution 
in $c_2(V)$ (which would lift the socle of realisable $\om$). 
We will discuss the nonsplit-condition (\ref{non-split condition})
(where $H_B$ has to be chosen orthogonal to $\al$). 
As $\eta(\eta-nc_1)>0$ we will just choose (for each $\eta$) a suitable $\al$ such that
$\al \cdot (\eta-n c_1)\leq 0$.

{\em The case of $B$ being a Hirzebruch surface ${\bf F_g}$}

Because of our assumption that $c_1$ is ample we have actually $g=0$ or $1$.
Let us take first $\al = (-1,1)$ such that $-\al^2=g+2$ 
and $H_B:=(e,f)=e(1,g+1)$ (which is ample as $f>ge$).
The evaluation $\al \cdot (\eta- n c_1)=(g+1)a-b-gn$ shows that one chooses
this $\al$ if the latter expression is $\leq 0$; if it is $>0$ one chooses just
$-\al$ instead of $\al$.

{\em The case of $B$ being a del Pezzo surface ${\bf dP_k}$}

Let us now also investigate the situation 
for the del Pezzo surfaces ${\bf dP_k}$
with $k=0, \dots, 8$ and $c_1=3l-\sum_k E_k$ 
(where $l$ is of course the class of the pull-back of the line).
On ${\bf P^2}$ the condition (\ref{orthogonality condition})
amounts, with $\al=a l$, to $a=0$, giving a contradiction.
The case of ${\bf dP_1}\cong {\bf F_1}$ is already settled.
Now in the case of $2\leq k\leq 8$
let $\al=kl-3\sum_{i=1}^k E_i$ such that $-\al^2=k(9-k)$ and $H_B=c_1$. Then one finds with $\eta= al+\sum b_iE_i$ that
$\al \cdot (\eta- nc_1)=ka+3\sum b_i$: if this expression is $\leq 0$
the choice of $\al$ was already successfull; otherwise one just takes $-\al$.

\subsection{The realizability of DRY classes by Chern classes of stable bundles}

Let us now look at the conditions under which we can realize a given
DRY class $c=\phi \si + \om$ as $c_2(V)$ for the extension bundle $V$ 
we have constructed. In the $\sigma$ term one just takes $\eta:=\phi$
where under our assumed condition $r\geq n$ the spectral requirement
that $\eta-nc_1$ is effective follows from the DRY porperty that $\phi-
\frac{n+r}{2}c_1$ is effective. Now let us compare the $\omega$ terms, i.e.~the
lower bound for $\om_{DRY}$ which follows from the DRY condition with 
the lower bound of the $\omega$ term in the expression $c_2(V)$ which comes from
a stable bundle.
Here one finds (with $n=3, \la = 1/2$; 
note that we had a different choice for $\al$ 
for ${\bf F_1}$ and ${\bf dP_1}$, respectively)
\beqa
\om_{DRY}&>&\frac{1}{2}\eta c_1+\frac{m}{6}c_1^2+\frac{m}{2}
+\frac{1}{2m}\sqrt{c_1^2}\sqrt{(\eta-\frac{m}{2}c_1)^2}\nonumber\\
&>&\frac{1}{2}\eta c_1+\frac{m}{6}c_1^2+\frac{m}{2}
=\left\{ \begin{array}{ll}
\frac{1}{2}\eta c_1+\;\; \frac{11}{6}\; m
\;\;\; \mbox{for} \;\;\, B={\bf F_g}\\
\frac{1}{2}\eta c_1+\frac{12-k}{6}m
\;\;\; \mbox{for} \;\; B={\bf dP_k}
\end{array} \right.\\
\om_V&\geq &-c_1^2+r+2-\frac{nr(n+r)}{2}\al^2\nonumber\\
&=&\left\{ \begin{array}{ll}
\;\;\;\;\; -8\;\; + \;\; m-1+\frac{3}{2}(m-3)m(g+2)
\;\;\;\;\;\;\;\;\;\;\;\;\;\:
\mbox{for} \;\; B={\bf F_g}\\
-(9-k)+m-1+\frac{3}{2}(m-3)mk(9-k)\;\;\;\;\;\;\;\;\;\;\;\;
\mbox{for} \;\; B={\bf dP_k}
\end{array} \right.
\eeqa

Thus one finds that, as the bound for $\om_{DRY}$ depends on $\eta$ while the
bound for $\om_V$ does not, only for finitely many choices of $\eta$ 
the corresponding bound for $\om_{DRY}$ does not lie above the bound
for $\om_V$ (keeping $m$ fixed). This finishes the proof of part (i) of Theorem \ref{main}.

For part (a) of (ii) note that for $\cN=m$ being $\equiv 2$ (mod 4) one can take $n=r$ (being odd) and use (still with
$\lambda =1/2$) the modified construction ({\ref{extv2}). In this case one gets 
\beqa
\om_V&\geq &-c_1^2+\frac{m}{2}+2-\al^2
\;=\; \left\{ \begin{array}{ll}
 \frac{m}{2}+g-4
\;\;\;\;\;\;\;\;\;\;\;\;
\mbox{for} \;\; B={\bf F_g}\\
\frac{m}{2}+k+1\;\;\;\;\;\;\;\;\;\;\;\;
\mbox{for} \;\; B={\bf dP_k}
\end{array} \right.
\eeqa
where this time for $dP_k$ we choose $\alpha=l-3E_1$ (which is also orthogonal to $c_1$)
with $-\alpha^2=8$.
To finish the proof for $B={\bf F_g}$ one just has to note that
$\frac{11}{6}m>\frac{m}{2}-4+g$. For $B={\bf dP_k}$ note that, using 
$\eta c_1> \frac{m}{2}(9-k)$ because $\eta-nc_1$ is ample, one has
$\omega_{DRY}>[\frac{5}{12}(9-k)+\frac{1}{2}]m$. For $k=1,\dots, 6$ 
the lower bound for $\omega_{DRY}$ is greater than the lower bound 
for $\omega_V$ because $\frac{5}{12}(9-k)m>k+1$. For $k=7,8$ only a small 
number of the possible choices for $m$ is excluded.

For part (b) of (ii) where one takes $n=r$ with $n$ even and $\lambda =1/2$ thus $B={\bf F_0}$
one notes that $\frac{11}{6}m>\frac{1}{2}m-4$.

Acknowledgements. B. A. thanks the SFB 647 for support. G. C. thanks the DFG for support 
in the grant CU 191/1-1 and the FU Berlin for hospitality.

\end{document}